\documentclass{article}

\usepackage{amsmath}
\usepackage{amsfonts}
\usepackage{amsthm}
\usepackage{mathrsfs}
\usepackage{amssymb}
\usepackage{tikz}
\usepackage{comment}
\usetikzlibrary{calc}
\usepackage{tikz-cd}
\usepackage{graphicx}
\usepackage[nodisplayskipstretch]{setspace}
\usepackage{varwidth}
\usepackage{dsfont}
\usepackage{mathtools}
\usetikzlibrary{graphs,decorations.pathmorphing,decorations.markings}

\DeclareMathOperator{\Proj}{Proj}

\DeclareMathOperator{\sort}{sort}

\DeclareMathOperator{\area}{area}

\DeclareMathOperator{\cont}{cont}
\DeclareMathOperator{\dinv}{dinv}
\DeclareMathOperator{\bounce}{bounce}

\renewcommand{\sc}{\mathrm{sc}}
\newcommand{\Oo}{\mathcal O}

\newcommand{\N}{\mathbb N}
\newcommand{\R}{\mathbb R}
\newcommand{\Z}{\mathbb Z}
\newcommand{\C}{\mathbb C}

\renewcommand{\H}{\mathbf H}

\newtheorem{theorem}{Theorem}[section]
\newtheorem{lemma}[theorem]{Lemma}
\newtheorem{corollary}[theorem]{Corollary}
\newtheorem{proposition}[theorem]{Proposition}

\theoremstyle{definition}
\newtheorem{definition}[theorem]{Definition}
\newtheorem{example}[theorem]{Example}

\theoremstyle{remark}

\numberwithin{theorem}{section}

\usepackage[margin=1.5in]{geometry}

\title{$q,t$-Catalan Measures}
\author{Ian Cavey}

\begin{document}

\maketitle

\begin{abstract}
We introduce the $q,t$-Catalan measures, a sequence of piece-wise polynomial measures on $\R^2$. These measures are defined in terms of suitable area, dinv, and bounce statistics on continuous families of paths in the plane, and have many combinatorial similarities to the $q,t$-Catalan numbers. Our main result realizes the $q,t$-Catalan measures as a limit of higher $q,t$-Catalan numbers $C^{(m)}_n(q,t)$ as $m\to\infty$. We also give a geometric interpretation of the $q,t$-Catalan measures. They are the Duistermaat-Heckman measures of the punctual Hilbert schemes parametrizing subschemes of $\C^2$ supported at the origin. 
\end{abstract}

\section{Introduction}\label{sec:Intro}

Introduced by Garsia and Haiman \cite{GHai}, the $q,t$-Catalan numbers $C_n(q,t)\in \N[q,t]$ are a sequence of polynomials that refine the sequence of Catalan numbers. These polynomials have connections to many areas of math including combinatorics, representation theory, symmetric function theory, and algebraic geometry \cite{Hag}.  \\

In combinatorics, the $q,t$-Catalan numbers are defined as weighted sums over the set of Dyck paths. A \textit{Dyck path of height} $n$ is lattice path from $(0,0)$ to $(n,n)$ consisting of north and east steps, both of unit length, that never goes strictly below the diagonal $y=x$. The set of all Dyck paths of height $n$ is denoted $\mathcal D_n$, and this set is enumerated by the Catalan number $C_n = \frac{1}{n+1}{2n\choose n}$. In terms of the area, dinv, and bounce statistics on Dyck paths (see Section \ref{sec:DyckPaths}), the $q,t$-Catalan numbers are given by either of the following equivalent formulas,
\begin{align}
C_n(q,t) & = \sum_{D\in \mathcal D_n}q^{\dinv(D)}t^{\area(D)} \label{intro:def1} \\
& = \sum_{D\in \mathcal D_n}q^{\area(D)}t^{\bounce(D)}.\label{intro:def2} 
\end{align}

More generally, the higher $q,t$-Catalan numbers $C_n^{(m)}(q,t)\in \N[q,t]$, also introduced by Garsia and Haiman \cite{GHai}, refine the higher Catalan numbers $C_n^{(m)} = \frac{1}{mn+1}{(m+1)n \choose n}$ and specialize to the ordinary $q,t$-Catalan numbers in the case $m=1$. In this paper, we mainly work with the combinatorial higher $q,t$-Catalan numbers introduced by Loehr \cite{L}. These polynomials are defined by formulas analogous to (\ref{intro:def1}) and (\ref{intro:def2}), by recording generalized area, dinv, and bounce statistics on $m$-Dyck paths of height $n$. An $m$-Dyck path of height $n$ is a lattice path from $(0,0)$ to $(mn,n)$ made up of north and east steps, both of unit length, that never goes strictly below the diagonal $y=\frac1mx$. In Section \ref{sec:DyckPaths} we review the precise definitions of these polynomials and the statistics on $m$-Dyck paths used to define them.\\

The higher $q,t$-Catalan numbers satisfy the joint symmetry property, $C^{(m)}_n(q,t) = C^{(m)}_n(t,q)$, which is not apparent from the combinatorial definition. There are many alternate definitions of these polynomials (see the introduction of \cite{LLL}), many of which are visibly symmetric. It is difficult, however, to show that the plainly symmetric algebraic definitions agree with the combinatorial ones. This was done for all $n$ in the case $m=1$ by Garsia and Haglund \cite{GHag}. More recently, Mellit's proof \cite{M} of the ``compositional $(km,kn)$-shuffle conjecture" \cite{BGLX} implies in particular that the combinatorial and algebraic definitions of the higher $q,t$-Catalan numbers agree for all $m$ and $n$. To our current knowledge, it remains an open problem to prove this joint symmetry combinatorially, even in the case $m=1$.\\

In this paper, we introduce continuous analogues of the $q,t$-Catalan numbers defined by formulas analogous to (\ref{intro:def1}) and (\ref{intro:def2}). We define a \textit{continuous Dyck path of height} $n$ to be a path from $(0,0)$ to $(n,n)$ consisting of north steps of unit length and east steps of arbitrary positive length that never goes below the diagonal $y=x$ (see Figure \ref{fig:coords}). Denote the set of continuous Dyck paths of height $n$ by $\mathcal D^{(\cont)}_n$. In Section \ref{sec:ContDyckPaths} we introduce real-valued area, dinv, and bounce statistics on continuous Dyck paths based on the corresponding statistics for $m$-Dyck paths. \\

Since the set of continuous Dyck paths is infinite, one cannot naively sum over $\mathcal{D}^{(\cont)}_n$ as in formulas (\ref{intro:def1}) and (\ref{intro:def2}). Instead, we formulate a measure-theoretic analogue. The set of all continuous Dyck paths can be naturally considered as a full-dimensional polytope $\mathcal{D}^{(\cont)}_n\subseteq \R^{n-1}$, on which all of our statistics form piece-wise linear, continuous maps $\mathcal{D}^{(\cont)}_n\to \R$. The \textit{$q,t$-Catalan measure} $\mu_n$ is then defined as the pushforward of Lebesgue measure from the polytope $\mathcal{D}^{(\cont)}_n$ to $\R^2$ by either of the following maps,
\[\begin{tikzcd}
\mathcal{D}^{(\cont)}_n \arrow[rrr,shift right,"\area\times \bounce"']\arrow[rrr,shift left, "\dinv\times\area"]&&& \R^2.
\end{tikzcd}\]

In Sections \ref{sec:measures} and \ref{sec:T}, we show that these two definitions agree by constructing a measure preserving transformation $T:\mathcal{D}^{(\cont)}_n\to\mathcal{D}^{(\cont)}_n$ such that $(\dinv\times\area)\circ T = \area\times\bounce$. The measure preserving property is the analogue of bijectivity for the analogous maps on $m$-Dyck paths, and in fact $T$ is not injective. We show directly that the $q,t$-Catalan measures are compactly supported, piece-wise polynomial measures on $\R^2$. In Example \ref{ex:4measure}, we explicitly compute the $q,t$-Catalan measure in the case $n=4$.\\

Our main result realizes the $q,t$-Catalan measure as a limit of higher $q,t$-Catalan numbers. Before making this precise, let us first illustrate the limit in the case $n=4$. Figure \ref{fig:disc density} depicts the polynomials $C_4^{(m)}(q,t)$ for several values of $m$, where darker shading of the cell $(i,j)$ indicates a larger coefficient on the $q^it^j$ term in the indicated polynomial. As $m\to\infty$, these discrete density functions can be normalized to converge to the continuous density function of the $q,t$-Catalan measure $\mu_4$ (see Example \ref{ex:4measure}).\\

\begin{figure}[h]
\begin{tabular}{ccc}
\includegraphics[width=0.31\textwidth]{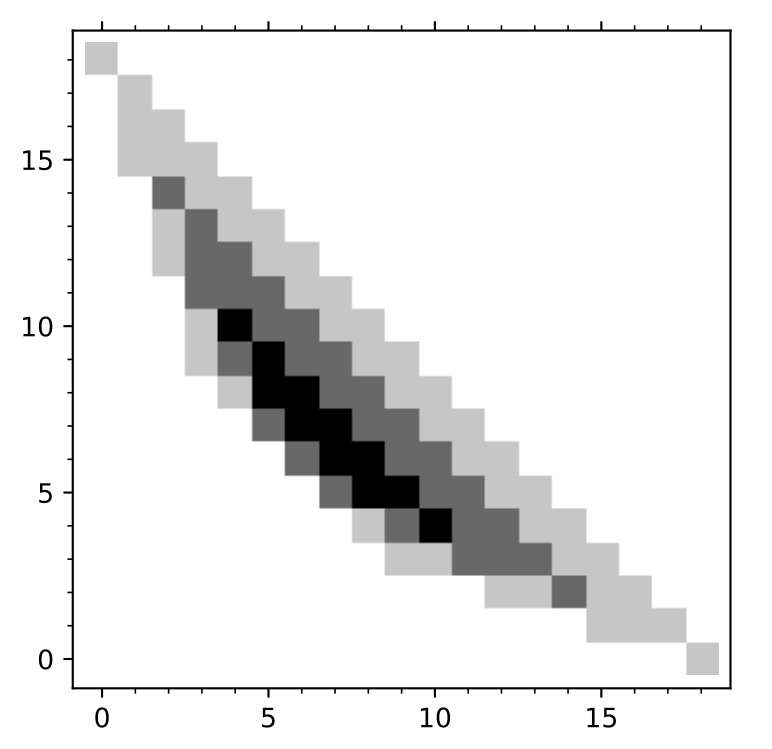}&
\includegraphics[width=0.3\textwidth]{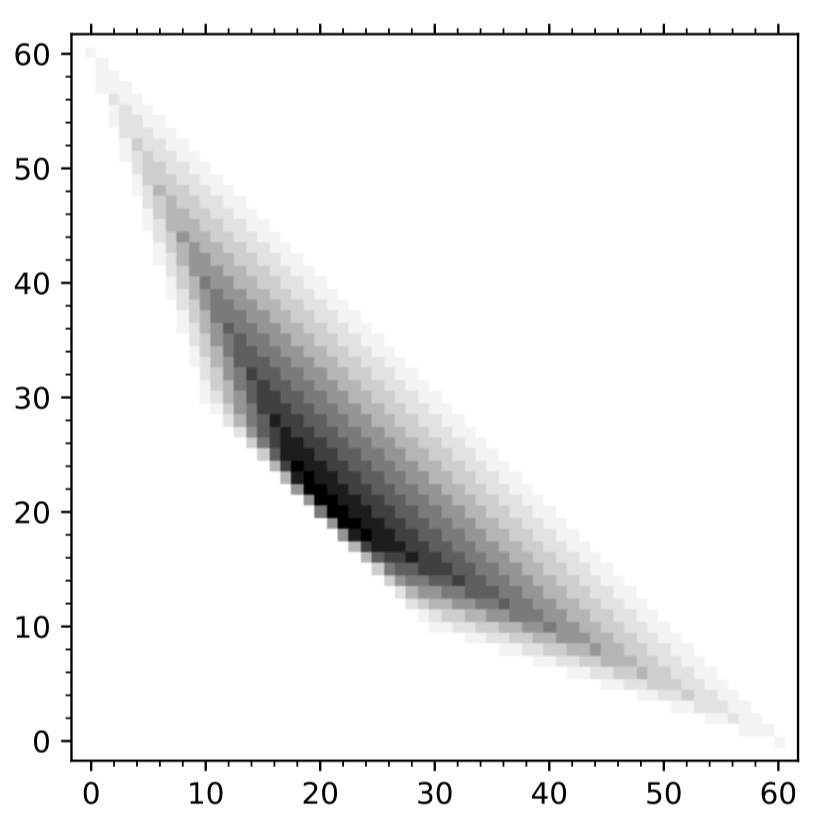}&
\includegraphics[width=0.31\textwidth]{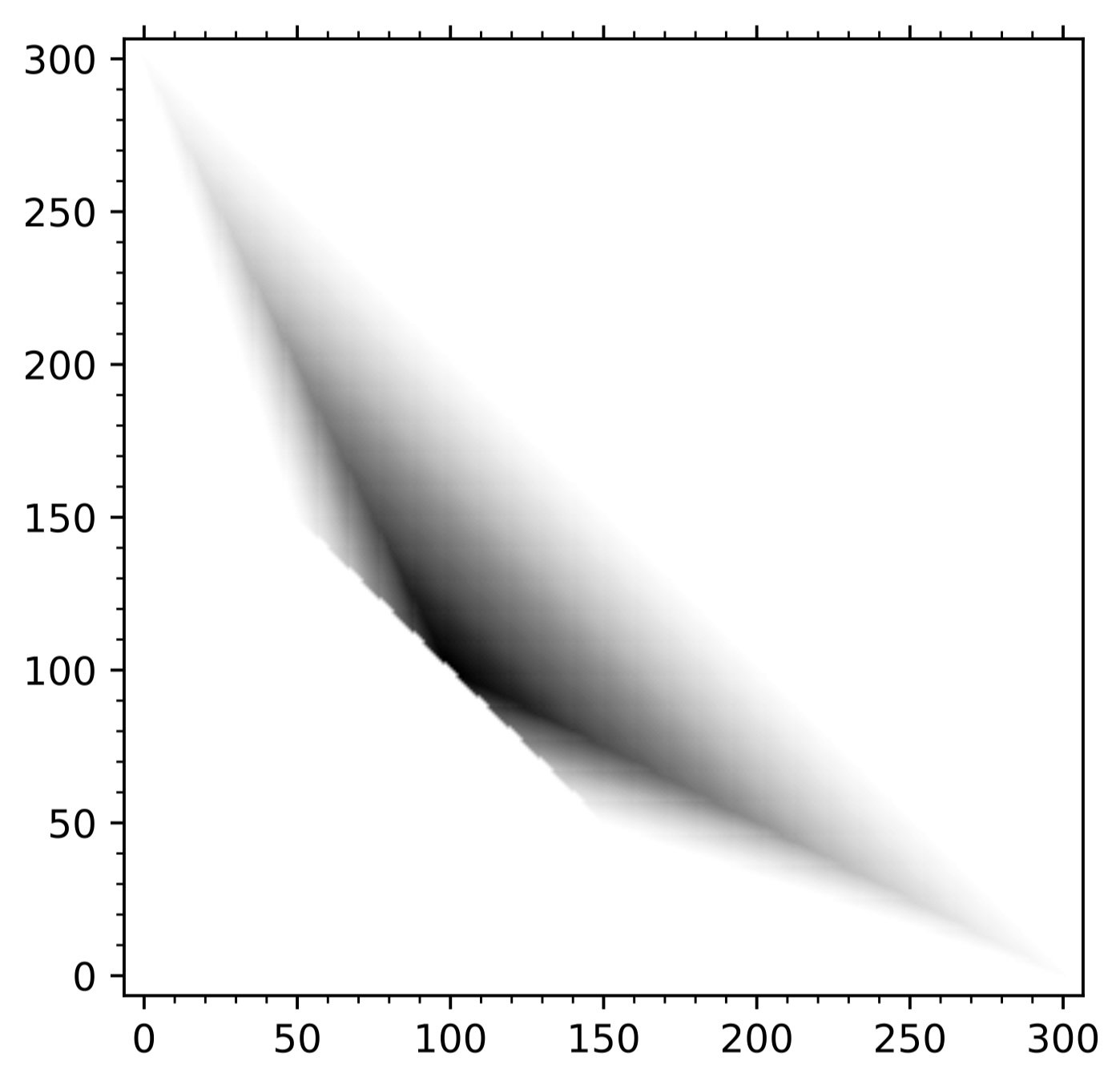} \\
$C_4^{(3)}(q,t)$ & $C_4^{(10)}(q,t)$ & $C_4^{(50)}(q,t)$\\
\end{tabular}
\caption{Discrete density functions of higher $q,t$-Catalan numbers}
    \label{fig:disc density}
\end{figure}

More precisely, encode $C^{(m)}_n(q,t)$ as a discrete measure on $\Z^2$ whose weight at $(i,j)$ is equal to the coefficient on the $q^it^j$ term of $C^{(m)}_n(q,t)$. Using the dinv, area formula for $C^{(m)}_n(q,t)$ and writing $\mathcal D^{(m)}_n$ for the set of $m$-Dyck paths of height $n$, this correspondence takes the form
\[ C^{(m)}_n(q,t) = \sum_{D\in \mathcal D^{(m)}_n}q^{\dinv(D)}t^{\area(D)}  \leftrightarrow \sum_{D\in \mathcal D^{(m)}_n}\delta_{(\dinv(D),\area(D))}, \]
where $\delta_{(a,b)}$ denotes a Dirac measure at the point $(a,b)$. We then normalize the measures by scaling their supports uniformly by a factor of $1/m$, and dividing the total weights by $m^{n-1}$. The following theorem, proved in Section \ref{sec:proof}, realizes the $q,t$-Catalan measures as a limit of these normalized discrete measures.

\begin{theorem}\label{introthm:limit}
For all $n\geq 1$, the $q,t$-Catalan measure $\mu_n$ is equal to the weak limit of measures on $\R^2$,
\[ \mu_n = \lim_{m\to\infty} \left(\frac{1}{m^{n-1}}\sum_{D\in \mathcal D^{(m)}_n}\delta_{\left(\frac{\dinv(D)}{m},\frac{\area(D)}{m}\right)}\right). \]
\end{theorem}

The proof of the theorem is based on a simple bijection between $m$-Dyck paths and those continuous Dyck paths whose horizontal step lengths all lie in $\frac1m\Z$: scale the $m$-Dyck path horizontally by a factor of $\frac1m$ so that it goes from $(0,0)$ to $(n,n)$ (see Figure \ref{fig:scaling}). The area, dinv, and bounce statistics on continuous Dyck paths were designed to agree with the corresponding (normalized) statistics on $m$-Dyck paths in the limit $m\to\infty$, from which we deduce the result. \\ 

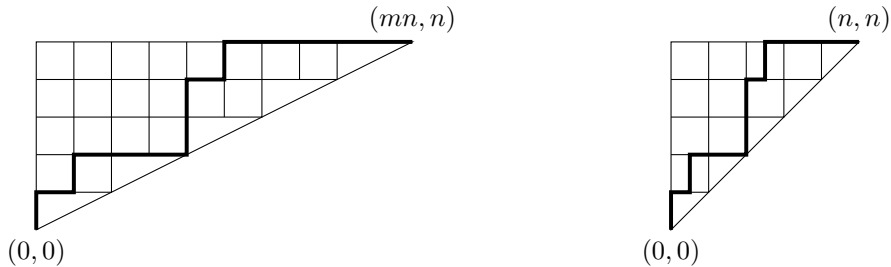
\begin{figure}[h]
    \centering
    \begin{tikzpicture}
    \draw (0,0) node[below] {$(0,0)$} -- (5,2.5) node[above] {$(mn,n)$} -- (0,2.5) -- cycle;
    \draw (0,0.5) -- (1,0.5) -- (1,2.5);
    \draw (0.5,0.5) -- (0.5,2.5);
    \draw (0,1) -- (2,1) -- (2,2.5);
    \draw (1.5,1) -- (1.5,2.5);
    \draw (0,1.5) -- (3,1.5) -- (3,2.5);
    \draw (2.5,1.5) -- (2.5,2.5);
    \draw (0,2) -- (4,2) -- (4,2.5);
    \draw (3.5,2) -- (3.5,2.5);
    \draw[line width = 1.5pt] (0,0) -- (0,0.5) -- (0.5,0.5) -- (0.5,1) -- (2,1) -- (2,2) -- (2.5,2) -- (2.5,2.5) -- (5,2.5);
    \end{tikzpicture}
    \hspace{2cm}
    \begin{tikzpicture}
    \draw (0,0) node[below] {$(0,0)$} -- (2.5,2.5) node[above] {$(n,n)$} -- (0,2.5) -- cycle;
    \draw (0,0.5) -- (0.5,0.5) -- (0.5,2.5);
    \draw (0,1) -- (1,1) -- (1,2.5);
    \draw (0,1.5) -- (1.5,1.5) -- (1.5,2.5);
    \draw (0,2) -- (2,2) -- (2,2.5);
    \draw[line width = 1.5pt] (0,0) -- (0,0.5) -- (0.25,0.5) -- (0.25,1) -- (1,1) -- (1,2) -- (1.25,2) -- (1.25,2.5) -- (2.5,2.5);
    \end{tikzpicture}
    \caption{An $m$-Dyck path and its corresponding continuous Dyck path}
    \label{fig:scaling}
\end{figure}

Since the higher $q,t$-Catalan numbers are known to be symmetric, this relationship implies a corresponding symmetry result for the $q,t$-Catalan measures.

\begin{corollary}
For all $n\geq 1$, $\mu_n$ is symmetric about the line $y=x$.
\end{corollary}

Similar to the (higher) $q,t$-Catalan numbers, the symmetry is not clear from the combinatorial definition. It would be interesting to find a direct proof of the symmetry of these measures, as it might give insight into the symmetry of the (higher) $q,t$-Catalan numbers as well. \\

Finally, in Section \ref{sec:geometric} we relate the $q,t$-Catalan measures to the geometry of Hilbert schemes. One of the alternate definitions of higher $q,t$-Catalan numbers, the \textit{geometric higher $q,t$-Catalan numbers} $GC^{(m)}_n(q,t)$, was introduced and studied by Haiman \cite{Hai}. Haiman showed that $GC^{(m)}_n(q,t)$ agrees with the algebraically defined higher $q,t$-Catalan numbers \cite{Hai}. Much later, Mellit connected the algebraic polynomials to the combinatorial ones as a consequence of the proof of the ``compositional $(km,kn)$-shuffle conjecture" \cite{M}. Together, these results imply that for all $n,m\geq 1$ we have
\[ GC^{(m)}_n(q,t) = C^{(m)}_n(q,t).\] 

There is a general construction in algebraic geometry to encode the asymptotics of such families of polynomials as $m\to\infty$, called the Duistermaat-Heckman measure \cite{BP}. Replacing $GC^{(m)}_n(q,t)$ by $C^{(m)}_n(q,t)$ in the definition of the Duistermaat-Heckman measure, one precisely recovers the limit expression for the $q,t$-Catalan measure expressed in Theorem \ref{introthm:limit}. This allows for the following geometric restatement of Theorem \ref{introthm:limit} in terms of the punctual Hilbert scheme $\H^n_0$ which parameterizes length $n$ subschemes of $\C^2$ supported at the origin.

\begin{theorem}\label{introthm:geo}
For all $n\geq 1$, the $q,t$-Catalan measure $\mu_n$ is equal to the Duistermaat-Heckman measure of the punctual Hilbert scheme $\H^n_0$.
\end{theorem}

This reinterpretation is explained in more detail in Section \ref{sec:geometric}, and the proof as outlined above is summarized in the following diagram.

\[\begin{tikzcd}
C^{(m)}_n(q,t) \arrow[rrr,leftrightarrow,"\text{(Haiman + Mellit)}"',"="] \arrow[dd,rightsquigarrow,"\lim\limits_{m \to \infty}"',"\text{(Theorem \ref{introthm:limit})}"] &&& GC^{(m)}_n(q,t)\arrow[dd,rightsquigarrow,"\lim\limits_{m \to \infty}"',"\text{(Definition)}"] \\
\\
(q,t\text{-Catalan measure}) \arrow[rrr,dashed,leftrightarrow] &&& (\text{DH measure})
\end{tikzcd}\]

\vspace{0.5cm}

\textit{Acknowledgements:} I am very grateful to Jim Haglund for explaining to me the connection between Mellit's results \cite{M} and the higher $q,t$-Catalan numbers. I also thank Dave Anderson for teaching me about Duistermaat-Heckman measures, and providing valuable comments and suggestions on various drafts of this project.

\section{Higher $q,t$-Catalan Numbers}\label{sec:DyckPaths}

In this section we review the combinatorics of $m$-Dyck paths and higher $q,t$-Catalan numbers, following \cite{L}. \\

An $m$-Dyck path of height $n$ is a path from $(0,0)$ to $(mn,n)$ consisting of north steps and east steps, both of unit length, that never goes strictly below the line $y=\frac1mx$. The set of all $m$-Dyck paths of height $n$ is denoted $\mathcal{D}^{(m)}_n$. The \textit{area vector} of an $m$-Dyck path $D$ is the vector $\mathbf{area}_m(D) = (a_0(D),\dots,a_{n-1}(D))$, where $a_i(D)$ denotes the number of complete boxes between $D$ and the line $y=\frac1mx$ in the $i$th row, indexed from $i=0$. The \textit{area} of $D$ is the total number of such boxes in all rows, $\area_m(D) = \sum_{i=0}^n a_i(D)$. For example, the path on the left of Figure \ref{fig:scaling} has area vector $(0,1,0,2,3)$ and area 6.\\

The \textit{dinv} statistic of an $m$-Dyck path $D$ is defined in terms of the area vector of $D$ by the formula $\dinv_m(D) = \sum_{i<j}\sc_m(a_i(D)-a_j(D))$, where $\sc_m:\Z\to\Z$ is the function
\[ \sc_m(p) = \begin{cases}
m+1-p & \text{if } 1\leq p\leq m,\\
m+p & \text{if } -m\leq p \leq 0,\\
0 & \text{otherwise}.
\end{cases}\]
For example, if $D$ is again the path on the left of Figure \ref{fig:scaling} one calculates 
\begin{align*} \dinv_2(D) & =  \sc_2(0-1) + \sc_2(0-0) + \sc_2(0-2)+\sc_2(0-3)+\sc_2(1-0) \\
& \hspace{0.5cm} +\sc_2(1-2)+\sc_2(1-3)+\sc_2(0-2)+\sc_2(0-3)+\sc_2(2-3)\\
& = 1+2+0+0+2+1+0+0+0+1\\
& = 7
\end{align*}

The \textit{bounce} statistic of an $m$-Dyck path $D$ is defined in terms of a secondary lattice path associated to $D$ called a \textit{bounce path}. The bounce path of $D$ is a lattice path from $(0,0)$ to $(mn,n)$ made up of an alternating sequence of north steps $v_0,v_1,\dots$ and east steps $h_0,h_1,\dots$. Starting from $(0,0)$, the bounce path first travels north until it hits an east step of $D$, and the distance traveled is labelled $v_0$. The bounce path then takes an east step of distance $h_0 := v_0$. Now suppose inductively that $v_0,\dots,v_{i-1}$ and $h_0,\dots,h_{i-1}$ have been defined. After these steps, the bounce path travels north until it hits an east step of $D$, and the vertical distance traveled is labelled $v_i$. The bounce path then takes an east step of distance $h_i := v_i+v_{i-1}+\cdots+v_{i-m+1}$, where any $v_j$ with $j<0$ is treated as zero. This process terminates when the bounce path reaches $(mn,n)$. The bounce statistic is defined in terms of the vertical steps $v_0,v_1,\dots$ of the bounce path by the formula $ \bounce_m(D) = \sum_{i\geq0} i\cdot v_i.$\\

For example, the steps in the bounce path of the Dyck path in Figure \ref{fig:scaling} are given by the following table from which the bounce statistic can be computed as $1\cdot 1+3\cdot 2 + 4\cdot 1 = 11$.

\begin{table}[h]
    \centering
    \begin{tabular}{|c|c|c|c|c|c|c|}
    \hline
         $i$  & $0$ & $1$ & $2$ & $3$ & $4$ & $5$ \\\hline
        $v_i$ & $1$ & $1$ & $0$ & $2$ & $1$ & $0$ \\\hline
        $h_i$ & $1$ & $2$ & $1$ & $2$ & $3$ & $1$ \\\hline
    \end{tabular}
\end{table}

There are two equivalent definitions of the combinatorial higher $q,t$-Catalan numbers in terms of these statistics on $m$-Dyck paths,
\begin{align} C^{(m)}_n(q,t) & = \sum_{D\in \mathcal D^{(m)}_n}q^{\dinv_m(D)}t^{\area_m(D)}\label{higherdef1} \\
& = \sum_{D\in \mathcal D^{(m)}_n}q^{\area_m(D)}t^{\bounce_m(D)}.\label{higherdef2}
\end{align}

Loehr \cite{L} constructs a bijection $\phi_m:\mathcal{D}^{(m)}_n\to \mathcal{D}^{(m)}_n$ such that $\dinv_m(D) = \area_m(\phi_m(D))$ and $\area_m(D) = \bounce_m(\phi_m(D))$ for all $D\in \mathcal{D}^{(m)}_n$. The existence of such a map implies that the two definitions (\ref{higherdef1}) and (\ref{higherdef2}) agree. Indeed,
\begin{align*} \sum_{D\in \mathcal D^{(m)}_n}q^{\dinv_m(D)}t^{\area_m(D)}
& = \sum_{D\in \mathcal D^{(m)}_n}q^{\area_m(\phi_m(D))}t^{\bounce_m(\phi_m(D))}\\
& = \sum_{D'\in \mathcal D^{(m)}_n}q^{\area_m(D')}t^{\bounce_m(D')}.
\end{align*}

Let us review the definition of this bijection. Let $D$ be an $m$-Dyck path. The bounce path of $\phi_m(D)\in \mathcal{D}^{(m)}_n$ will be given by the sequence $v_0,h_0,v_1,h_1,\dots,v_s,h_s$, where $v_i$ is equal the number of occurrences of $i$ in the area vector of $D$, and $h_i = v_i+\cdots+v_{i-m+1}$. Let $p_0,p_1,p_2,\dots,p_{s+1}$ be the sequence of points on the bounce path where $p_0=(0,0)$, $p_{s+1} = (mn,n)$, and $p_i$ is the end point of the partial bounce path $v_0,h_0,\dots,v_{i-1}$ for $1\leq i\leq s$. The path $\phi_m(D)$ will pass through all of the $p_i$'s, and the rule for drawing the portion of the path $\phi_m(D)$ from $p_i$ to $p_{i+1}$ is as follows. Read through the area vector of $D$, $(a_0(D),\dots,a_{n-1}(D))$, from left to right. Every time the symbol $i$ is seen, $\phi_m(D)$ takes a unit step north. Every time a symbol in $\{i-1,\dots,i-m\}$ is seen, $\phi_m(D)$ takes a unit step east.\\

Loehr shows that $\phi_m(D)$ is well defined (and the bounce path $\phi_m(D)$ is as claimed), and that the map sends dinv to area to bounce and is a bijection. The existence of such a map does not immediately imply the conjectural joint symmetry property $C^{(m)}_n(q,t) = C^{(m)}_n(t,q)$. It does, however, imply the weaker statement that these three statistics have the same univariate distributions,
\[ \sum_{D\in \mathcal D^{(m)}_n}q^{\dinv_m(D)} = \sum_{D\in \mathcal D^{(m)}_n}q^{\area_m(D)} = \sum_{D\in \mathcal D^{(m)}_n}q^{\bounce_m(D)}. \]
These identities can be obtained by specializing (\ref{higherdef1}) and (\ref{higherdef2}) to $q=1$ or $t=1$ separately.\\

The joint symmetry property for higher $q,t$-Catalan measures follows from Millet's results in \cite{M} proving the compositional $(km,kn)$-shuffle conjecture. This is discussed in more detail in Section \ref{sec:geometric}. For now, we state the joint symmetry as a theorem to refer to later.

\begin{theorem}[Mellit \cite{M}]\label{thm:sym}
The (combinatorial) higher $q,t$-Catalan numbers satisfy the joint symmetry property, 
\[ C^{(m)}_n(q,t) = C^{(m)}_n(t,q), \]
for all $n,m\geq 1$.
\end{theorem}

\section{Continuous Dyck Paths and the $q,t$-Catalan Measures}\label{sec:ContDyckPaths}

\subsection{Continuous Dyck Path Statistics}\label{sec:contstats}

A \textit{continuous Dyck path of height} $n$ is a path from $(0,0)$ to $(n,n)$ consisting of north steps of unit length and east steps of arbitrary positive length that never goes below the diagonal $y=x$. To avoid having multiple representations of the same path, we assume that continuous Dyck paths never contain two or more consecutive east steps. We denote the set of continuous Dyck paths of height $n$ by $\mathcal D^{(\cont)}_n$.\\

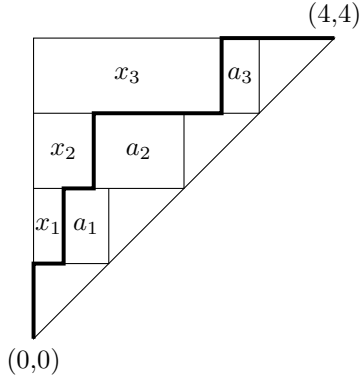
\begin{figure}[h]
    \centering
    \begin{tikzpicture}
    \draw (0,0)--(0,4)--(4,4)--cycle;
    \draw (0,1)--(1,1)--(1,2);
    \draw (0,2)--(2,2)--(2,3);
    \draw (0,3) -- (3,3) --(3,4);
    \draw[line width = 1.5pt] (0,0)--(0,1) -- (0.4,1)--(0.4,2) --(0.8,2)--(0.8,3)--(2.5,3)--(2.5,4)--(4,4);
    \node at (0.2,1.5) {$x_1$};
    \node at (0.7,1.5) {$a_1$};
    \node at (0.4,2.5) {$x_2$};
    \node at (1.4,2.5) {$a_2$};
    \node at (1.25,3.5) {$x_3$};
    \node at (2.75,3.5) {$a_3$};
    \node[below] at (0,0) {(0,0)};
    \node[above] at (4,4) {(4,4)};
    \end{tikzpicture}
    \caption{A continuous Dyck path of height $4$}
    \label{fig:coords}
\end{figure}

Let $D$ be a continuous Dyck path of height $n$, and for $i=0,\dots,n-1$ define $x_i(D)$ to be the $x$-coordinate of the $i$th north step of the path $D\in \mathcal{D}^{(\cont)}_n$ indexed from $i=0$. It follows from the definition that continuous Dyck paths must start with at least one north step, and therefore $x_0(D)$ is always zero. It will be convenient to include this leading zero as one of the coordinates regardless. The \textit{area vector of} $D$ is the vector $\mathbf{area}(D) = (a_0(D),\dots,a_{n-1}(D))\in \R^n$, where $a_i(D) = i-x_i(D)$ for all $i=0,\dots,n-1$ (see Figure \ref{fig:coords}). The \textit{area} of $D$ is defined by
\[ \area(D) = \sum_{i=0}^{n-1}a_i(D). \]
For example, the path $D$ depicted in Figure \ref{fig:coords} has north steps at the $x$-coordinates $(x_0,x_1,x_2,x_3) = (0,0.4,0.8,2.5)$. The area vector of $D$ is therefore $(0,0.6,1.2,0.5)$, and so $\area(D) = 2.3$.\\

The \textit{dinv} statistic of a continuous Dyck path $D$ is also defined in terms of its area vector by the formula
\[ \dinv(D) = \sum_{0\leq i<j\leq n-1} \sc(a_i(D)-a_j(D)), \]
where $\sc(x) = \max\{ 1-|x|,0\}$.\\

The dinv statistic of the path $D$ depicted in Figure \ref{fig:coords} can be calculated as
\begin{align*} \dinv(D) & = \sc(0.6)+\sc(1.2)+\sc(0.5)+\sc(1.2-0.6)+\sc(0.5-0.6)+\sc(0.5-1.2)\\
& = 0.4 + 0 + 0.5 + 0.4 + 0.9 + 0.3\\
& = 2.5
\end{align*}

The final statistic on continuous Dyck paths is defined in terms of a parametrization of the path that we call the \textit{bounce parametrization}. This parametrization is analogous to the bounce path of an $m$-Dyck path (see Section \ref{sec:limits} for the precise relationship). Given a continuous Dyck path $D$, the bounce parametrization of $D$ travels along the path $D$ from $(0,0)$ to $(n,n)$ on some time interval $0\leq t\leq t_{\max}(D)$. The parametrization is uniquely determined by the following two rules:
\begin{enumerate}
    \item Whenever the parametrization reaches the bottom of a north step of $D$, say of length $v$, the parametrization instantaneously takes $v$ (unit length) north steps to the top of the step.
    \item The parametrization travels continuously along the east steps of $D$ as $t$ increases, and the horizontal speed of the parametrization at any given time $t$ is equal to total number of (unit length) north steps taken by the parametrization in the time interval $[t-1,t]$.
\end{enumerate}

We define $b_i(D)$ for $i=0,\dots,n-1$ to be the time at which the bounce parametrization takes its $i$th (unit length) north step, indexed from $i=0$. The bounce parametrization is entirely determined by the vector $(b_0(D),\dots,b_{n-1}(D))$. Indeed, by the fundamental theorem of calculus the horizontal position of the bounce parametrization is computed by the function 
\[ r_D(t) = \int_0^t \sum_{i=0}^{n-1}\mathds{1}_{[b_i(D),b_i(D)+1]}\mathrm{d}s. \]
We may therefore have equivalently defined the coordinates $b_i(D)$ to be the unique nondecreasing sequence $b_0(D)\leq \cdots \leq b_{n-1}(D)$ such that the function above satisfies $r_D(b_i(D))= x_i(D)$ for all $i=0,\dots,n-1$. We call $\mathbf{bounce}(D) = (b_0(D),\dots,b_{n-1}(D))$ the \textit{bounce vector} of $D$, and define the bounce of $D$ by the formula
\[ \bounce(D) = \sum_{i=0}^{n-1} b_i(D). \]

Returning again to the path $D$ depicted in Figure \ref{fig:coords}, the bounce parametrization of $D$ can be described as follows:
\begin{itemize}
    \item At time $t=0$, the bounce parametrization of $D$ takes a north step from $(0,0)$ to $(0,1)$, and the first coordinate of the bounce vector is recorded as $b_0=0$. The path then begins moving east at speed $1$.
    \item At time $t=0.4$, the parametrization reaches $(0.4,1)$ where $D$ has a north step, so the next coordinate of the bounce vector is $b_1=0.4$. The parametrization takes a north step up to $(0.4,2)$ and continues moving east, now at speed $2$.
    \item At time $t=0.6$ the parametrization reaches the point $(0.8,2)$ where $D$ has its next north step, so the next coordinate of the bounce vector is $b_2=0.6$. The parametrization then takes a north step up to $(0.8,3)$ and continues moving east, now at speed $3$.
    \item At time $t=1$, the parametrization is at the point $(2,3)$ and one unit time has passed since the first north step of the bounce parametrization, so the speed slows down to $2$.
    \item At time $t=1.25$, the parametrization reaches $(2.5,3)$ where $D$ has its final north step, so the final coordinate of the bounce vector is $b_3=1.25$. The parametrization takes its final north step to $(2.5,4)$ and continues moving east, now at speed $3$ again.
    \item At times $t=1.4$, $1.6$, and $2.25$ the parametrization slows to speed $2$, then speed $1$, and then stops. These occur at the points $(2.95,4),$ $(3.35,4)$, and $(4,4)$ respectively.
\end{itemize}
The bounce vector of $D$ is therefore $(0,0.4,0.6,1.25)$, and so the bounce of $D$ is $2.25$. 

\subsection{Area and Bounce Polytopes}

The set of all area vectors of continuous Dyck paths of height $n$ forms an $(n-1)$-dimensional polytope, 
\[ A_n = \{ (a_0,\dots,a_{n-1})\in \R^n \,|\, a_0=0, \text{ and } 0\leq a_{i+1}\leq a_i+1 \text{ for all }i\geq 0\}. \]
Similarly, the set of all bounce vectors forms another $(n-1)$-dimensional polytope,
\[ B_n = \{ (b_0,\dots,b_{n-1})\in \R^n\,\big|\, b_0=0,\text{ and } b_i\leq b_{i+1}\leq b_i+1 \text{ for all } i\geq 0 \}. \] 
We call $A_n$ and $B_n$ the \textit{area polytope} and \textit{bounce polytope} respectively. When convenient, we will consider $A_n$ and $B_n$ as full dimensional polytopes in $\R^{n-1}$ by forgetting the first coordinate (which is identically zero in both cases). 

\begin{proposition}\label{area vs bounce}
The area vector $(a_0(D),\dots,a_{n-1}(D))$ and bounce vector $(b_0(D),\dots,b_{n-1}(D))$ of a continuous Dyck path $D\in \mathcal{D}^{(\cont)}_n$ are related by the formulas
\[ a_j(D) =  \sum_{i=0}^{j-1} \sc(b_j(D)-b_i(D)) \]
for all $j=0,\dots,n-1$.
\end{proposition}

We suppress the path $D$ in the notation throughout the proof, writing $a_j$ for $a_j(D)$, etc.

\begin{proof}
The bounce vector of $D$, $(b_0,\dots,b_{n-1})$, satisfies $r(b_j)=x_j = j-a_j$ for all $j=0,\dots,n-1$, where $r(t)$ is the function
\[ r(t) = \int_0^t \sum_{i=0}^{n-1}\mathds{1}_{[b_i,b_i+1]}ds. \] 

Plugging in $t=b_j$, we have
\[ j-a_j = \sum_{i=0}^{n-1}\int_0^{b_j} \mathds{1}_{[b_i,b_i+1]}dt.\]
Since $b_0\leq\cdots\leq b_{n-1}$, all of the terms $\int_0^{b_j}\mathds{1}_{[b_i,b_i+1]}ds$ for $i\geq j$ are zero. The remaining terms are either equal to $1$ (if $b_i+1< b_j$) or $b_j-b_i$ (if $b_j\in [b_i,b_i+1]$). These cases can be expressed succinctly using the function $1-\sc(x) = \min\{|x|,1\}$, giving the relation
\[ j-a_j= \sum_{i=0}^{j-1}(1-\sc(b_j-b_i)) = j - \sum_{i=0}^{j-1}\sc(b_j-b_i). \]
Solving for $a_j$ yields the desired formula.
\end{proof}

Area and bounce vectors provide two ways to parameterize the set of continuous Dyck paths by polytopes $A_n\leftrightarrow \mathcal{D}^{(\cont)}_n \leftrightarrow B_n$, and Proposition \ref{area vs bounce} can be viewed as a description of the change of coordinates map $B_n\to A_n$. Proposition \ref{area vs bounce} shows in particular that the bijection $B_n\to A_n$ defined by sending the bounce vector of a path $D$ to its area vector is piece-wise linear. \\

As outlined in the introduction, we aim to equip $\mathcal{D}^{(\cont)}_n$ with a measure by identifying it with an $(n-1)$-polytope. The area and bounce polytopes provide two candidates for this job, and the resulting measures on $\mathcal{D}^{(\cont)}_n$ are different as long as $n\geq 3$. Equivalently, the bijection $B_n\to A_n$ described in Proposition \ref{area vs bounce} is not measure preserving. This is necessarily the case because the defining inequalities show that $B_n$ is strictly contained in $A_n$. For our purposes, the more natural measure on $\mathcal{D}^{(\cont)}_n$ is the one from $A_n$.

\begin{comment}
\begin{lemma}\label{lemma:intpoints}
The set of integral points $A_n\cap \Z^n$ is precisely the set of area vectors of Dyck path of height $n$. More generally, $A_n\cap \frac1m\Z^n$ is the set of normalized area vectors of $m$-Dyck paths of height $n$,

\[ A_n \cap \frac{1}{m}\Z^n = \left\{ \left(\frac{a_0(D)}{m},\dots,\frac{a_{n-1}(D)}{m}\right)\, \bigg|\, D\in \mathcal{D}^{(m)}_n \right\}. \]
\end{lemma}
\end{comment}

\begin{definition}
Let $\lambda_n$ denote the restriction of the Lebesgue measure from $\R^{n-1}$ to the full-dimensional polytope $A_n\subseteq \{0\}\times\R^{n-1}=\R^{n-1}.$ By abuse of notation, we also consider $\lambda_n$ as a measure on $\mathcal{D}^{(\cont)}_n$ via the bijection identifying a continuous Dyck path $D$ with its area vector.
\end{definition}

\subsection{$q,t$-Catalan Measures}\label{sec:measures}

The \textit{$q,t$-Catalan measure} $\mu_n$ is defined to be the pushforward of $\lambda_n$ from $\mathcal{D}^{(\cont)}_n$ to $\R^2$ by either map 
\begin{equation}
\begin{tikzcd}
\mathcal{D}^{(\cont)}_n \arrow[rrr,shift right,"\area\times \bounce"']\arrow[rrr,shift left, "\dinv\times\area"]&&& \R^2.\label{def:qtmeasure}
\end{tikzcd}
\end{equation}

In Section \ref{sec:T} we define a map $T:\mathcal{D}^{(\cont)}_n\to\mathcal{D}^{(\cont)}_n$ and show that it is measure preserving in the sense that $T_*(\lambda_n) = \lambda_n$ (Proposition \ref{prop:Tmeasure}), and that $(\area\times\bounce)\circ T = (\dinv\times\area)$ as functions $\mathcal{D}^{(\cont)}_n\to \R^2$ (Lemma \ref{lemma:Tstats}). For now, we simply assert that there exists such a map and use it in the following calculation to show that the two formulas for the $q,t$-Catalan measure agree:
\begin{align*} 
(\dinv\times\area)_*\,(\lambda_n) & = ((\area\times\bounce)\circ T)_*\, (\lambda_n) \\
& = (\area\times\bounce)_*\, (T_*\, (\lambda_n)) \\
& = (\area\times\bounce)_*\,(\lambda_n). 
\end{align*}

\begin{proposition}
For all $n\geq 1$, the $q,t$-Catalan measure $\mu_n$ is compactly supported and has total weight $\mu_n(\R^2) = \frac{n^{n-2}}{(n-1)!}$.
\end{proposition}

\begin{proof}
Identifying $\mathcal{D}^{(\cont)}_n\leftrightarrow A_n$, $\mu_n$ is the pushforward of Lebesgue measure on $A_n$ by the map $\dinv\times\area:A_n\to \R^2$. The $q,t$-Catalan measure $\mu_n$ is supported on the image of this map, which is compact.\\

Since pushforwards preserve the total weight of a measure, $\mu_n(\R^2)$ is equal to the Lebesgue measure, i.e. volume, of $A_n$ considered as a polytope in $\R^{n-1}$. One way to compute this volume is using the Ehrhart polynomial, $P_{A_n}(m)$, whose value at each positive integer $m$ is equal to the number of integer points in the dilation $mA_n$. Equivalently, $P_{A_n}(m)$ counts the number of $1/m$-integer points in $A_n$. But $1/m$-integer points in $A_n$ are precisely the area vectors of $m$-Dyck paths of height $n$ scaled by $1/m$, as is illustrated in Figure \ref{fig:scaling}. It follows that
\[ P_{A_n}(m) = C^{(m)}_n = \frac{1}{mn+1}{(m+1)n\choose n} = \frac{(mn+n)\cdots(mn+2)}{n!} = \frac{n^{n-2}}{(n-1)!}m^{n-1}+O(m^{n-2}) \]
The general theory of Ehrhart polynomials implies that the leading coefficient of this polynomial, $\frac{n^{n-2}}{(n-1)!}$, is equal to the volume of $A_n$.
\end{proof}

\begin{proposition}
For $n\geq 3$, $\mu_n$ is absolutely continuous with respect to Lebesgue measure and its density function is piece-wise polynomial of degree $n-3$. 
\end{proposition}

\begin{proof}
Let $P$ be a polytope equipped with a linear map $\pi:P\to \R^k$ such that $\pi(P)\subseteq \R^k$ is full-dimensional. Then the pushforward of Lebesgue measure from $P$ to $\R^k$ is absolutely continuous with respect to Lebesgue measure on $\R^k$ and its density function is piece-wise polynomial of degree $\dim(P)-k$. When $n\geq 3$, the piece-wise linear map $(\dinv\times\area)A_n\to \R^2$ can be broken up into a sum of such terms by considering the regions on which the projection is linear separately. Each term is the projection from an $(n-1)$-dimensional polytope to $\R^2$, so the density functions are piece-wise polynomial of degree $n-3$. The $q,t$-Catalan measure is the sum of these, which completes the proof.
\end{proof}

For $n\geq 3$, let $f_n:\R^2\to \R$ denote the piece-wise polynomial density function for the $q,t$-Catalan measure $\mu_n$. At any given point $p\in \R^2$, $f_n(p)$ can be computed exactly by computing the volume of the fiber of the projection $A_n\to \R^2$ over $p$. Furthermore, the images of the edges of $A_n$ subdivide $\R^2$ into regions on which $f_n$ is given by a single polynomial of degree $n-3$. One can therefore compute the entire density function $f_n$ for any given $n$ by interpolating on each of the regions described above.

\begin{example}\label{ex:4measure}
The following figure illustrates the piece-wise linear density function $f_4$ for the $q,t$-Catalan measure $\mu_4$. In each of the three regions below, $f_4$ is equal to the indicated linear function. Outside of these regions, $f_4$ is zero.
\[\begin{tikzpicture} % n=4 Catalan measure
\draw[help lines, color=gray!60, dashed] (0,0) grid (6.9,6.9);

\draw [-{Latex[length=1.5mm]}] (0,0) -- (7,0);
\draw [-{Latex[length=1.5mm]}] (0,0) -- (0,7);

\node[below] at (1,0) {$1$};
\node[below] at (2,0) {$2$};
\node[below] at (3,0) {$3$};
\node[below] at (4,0) {$4$};
\node[below] at (5,0) {$5$};
\node[below] at (6,0) {$6$};
\node[left] at (0,1) {$1$};
\node[left] at (0,2) {$2$};
\node[left] at (0,3) {$3$};
\node[left] at (0,4) {$4$};
\node[left] at (0,5) {$5$};
\node[left] at (0,6) {$6$};

\node[circle,fill,inner sep = 1pt] (A) at (6,0) {};
\node[circle,fill,inner sep = 1pt] (B) at (3,1) {};
\node[circle,fill,inner sep = 1pt] (C) at (2,2) {};
\node[circle,fill,inner sep = 1pt] (D) at (1,3) {};
\node[circle,fill,inner sep = 1pt] (E) at (0,6) {};
\draw (A) -- (B) -- (C) -- (D) -- (E) -- (A) -- (C) -- (E);

\draw [-{Latex[length=1.5mm]}] (2.5,4.5) node[right]{$f_4(x,y) = 3x+y-6$} to [bend left=0] (1.1,3.3);
\draw [-{Latex[length=1.5mm]}] (3.5,3.5) node[right]{$f_4(x,y) = 6-x-y$} to [bend right=0] (2.7,2.5);
\draw [-{Latex[length=1.5mm]}] (4.5,2.5) node[right]{$f_4(x,y) = x+3y-6$} to [bend right=0] (3.3,1.1);
\end{tikzpicture}\]
\end{example}

In the above example, one can see that the $q,t$-Catalan measure is symmetric about the line $y=x$. In other words, the density function satisfies $f_4(x,y)=f_4(y,x)$. This symmetry holds for all $n$, and is not apparent from the combinatorial definitions. In Corollary \ref{cor:symmetry}, we deduce this symmetry from the joint symmetry of the higher $q,t$-Catalans expressed in Theorem \ref{thm:sym}. \\

An interesting problem would be to directly show that the $q,t$-Catalan measures are symmetric without using the corresponding result for the higher $q,t$-Catalan numbers. One could hope to find a measure preserving involution on $\mathcal{D}^{(\cont)}_n$ that switches area and dinv, or area and bounce, for example. So far, we have not been able to find such a map. 

\subsection{A Measure Preserving Transformation on Continuous Dyck Paths}\label{sec:T}

Let us now define the map $T:\mathcal{D}^{(\cont)}_n\to \mathcal{D}^{(\cont)}_n$ used in the previous section to show the equivalence between the two definitions of the $q,t$-Catalan measures.\\

For a continuous Dyck path $D$, define $T(D)$ to be the unique path whose bounce vector $(b_0(T(D)),\dots,b_{n-1}(T(D)))$ is equal to the area vector of $D$ $(a_0(D),\dots,a_{n-1}(D))$ up to a permutation of the coordinates. In other words, sort the area vector of $D$ in to weakly increasing order, and declare that $T(D)$ is the path whose bounce vector is equal to the sorted area vector of $D$.

\begin{lemma}\label{lemma:Tstats}
For all $D\in \mathcal D^{(\cont)}_n$, we have $\dinv(D) = \area(T(D))$, and $\area(D) = \bounce(T(D))$.
\end{lemma}

For example, let us return to the path $D$ depicted in Figure \ref{fig:coords} which has area vector $(0,0.6,1.2,0.5)$. The area of $D$ is 2.3, and we computed in Section \ref{sec:contstats} the statistic $\dinv(D)=2.5$. By definition, $T(D)$ is the path whose bounce vector is $(0,0.5,0.6,1.2)$, and Proposition \ref{area vs bounce} allows us to compute the area vector of $T(D)$ to be $(0,0.5,1.3,0.7)$. We therefore have $\area(T(D)) = 2.5 = \dinv(D)$ and $\bounce(T(D)) = 2.3 = \area(D)$, as claimed.

\begin{proof}
The area vector of $D$ and the bounce vector of $T(D)$ are the same up to a permutation, so 
\[ \area(D) = \sum_{i=0}^{n-1} a_i(D) = \sum_{i=0}^{n-1} b_i(T(D)) = \bounce(T(D)). \]
For the other claim we use Proposition \ref{area vs bounce} which expresses the coordinates of the area vector in terms of the bounce vector. Expanding the terms $a_j(T(D))$ using those formulas, we obtain
\[ \area(T(D)) = \sum_{j=0}^{n-1} a_j(T(D)) = \sum_{0\leq i<j\leq n-1}\sc(b_j(T(D))-b_i(T(D))). \]
Since $\sc(x) = \max\{ 1-|x|,0\}$ is even, the function $(v_0,\dots,v_{n-1})\mapsto \sum_{i<j}\sc(v_j-v_i)$ does not depend on the order of the coordinates of the input. The bounce vector of $T(D)$ and the area vector of $D$ are the same up to a permutation, so the previous sum is equal to
\[  \sum_{0\leq i<j\leq n-1}\sc(a_j(D)-a_i(D)) = \dinv(D) \]
as claimed.
\end{proof}

The map $T$ is surjective but not injective, and in general $|T^{-1}(D)|$ depends on $D$. Despite this apparent complication, we have the following result.

\begin{proposition}\label{prop:Tmeasure}
The map $T:\mathcal D^{(\cont)}_n\to \mathcal D^{(\cont)}_n$ is measure preserving, in the sense that $\lambda_n(U) = \lambda_n(T^{-1}U)$ for all measureable sets $U\subseteq \mathcal D^{(\cont)}_n$.
\end{proposition}

\begin{proof}
Consider $A_n$ and $B_n$ as full-dimensional polytopes in $\R^{n-1}$ by dropping the leading zeroes. We also identify $\mathcal D^{(\cont)}_n\leftrightarrow A_n$ so that we may consider $T$ as a function $A_n\to A_n$. This map can then be decomposed as
\[\begin{tikzcd}
A_n \arrow[r,"\sort"] & B_n \arrow[r,"f"] & A_n,
\end{tikzcd}\]
where the first map sorts vectors $(a_1,\dots,a_{n-1})\in A_n$ into weakly increasing order, and $f$ sends the bounce vector of a path $D$ to the area vector of the same path $D$. Proposition \ref{area vs bounce} provides an expression of the vector $(a_1,\dots,a_{n-1}) = f(b_1,\dots,b_{n-1})$ in terms of the function $\sc(x) = \max\{1-|x|,0\}$, namely
\[ a_j =  \sum_{i=0}^{j-1} \sc(b_j-b_i), \]
for all $j=1,\dots,n-1$. In particular, this description shows that $f$ is piece-wise linear.\\

Let $\mathbf{b} = (b_1,\dots,b_{n-1})\in B_n$. We wish to compute the Jacobian determinant of $f$ at $\mathbf{b}$, so we assume that $f$ is given by a single linear function in a neighborhood of $\mathbf{b}$. Examining the explicit description of $f$ given above, this means that all of the coordinates of $\mathbf{b}$ are distinct and nonzero, and there is no pair of indices $i<j$ such that $b_j = b_i+1$. \\

Claim: The magnitude of the Jacobian determinant of $f$ at $\mathbf{b}$ is equal to $|\sort^{-1}(\mathbf{b})|$, and this common value is precisely
\[ d(\mathbf{b}) = \prod_{j = 1}^{n-1} \#\{ i=0,\dots,j-1\,|\, b_j-b_i<1 \}. \]

First let us see how this claim implies the proposition. Since $f$ is a bijection and $\sort$ has Jacobian determinant $\pm1$, this claim implies that the preimage $T^{-1}(f(\mathbf{b}))$ also consists of $d(\mathbf{b})$ points, and the Jacobian determinant of $T$ is $\pm d(\mathbf{b})$ at all of them. It follows that $T$ preserves the measure of sufficiently small sets containing $f(\mathbf{b})$. But this applies to a dense set of points in $A_n$, which implies that $T$ is globally measure preserving.\\

Proof of Claim: To compute the Jacobian determinant of $f$ at $\mathbf{b}$, we expand the formula for the coordinates of $f(\mathbf{b})$ using the definition of $\sc(x)$ to obtain
\[ a_j = \sum_{\substack{i=0,\dots,j-1\\ b_j-b_i<1}} (1-b_j+b_i), \]
for all $j=1,\dots,n-1$. The $a_j$ coordinate in the expression only depends on $b_1,\dots,b_j$, so the Jacobian of $f$ at $\mathbf{b}$ is lower triangular. Furthermore, for all $j=1,\dots,n-1$ the $j$th entry on the diagonal of the Jacobian is $-\#\{ i=0,\dots,j-1\,|\, b_j-b_i<1 \},$ the coefficient on the $b_j$ term in the expression for $a_j$. The Jacobian determinant of $f$ at $\mathbf{b}$ is the product of these diagonal entries, and therefore has magnitude $d(\mathbf{d})$ as claimed.\\

Finally, we count $|\sort^{-1}(\mathbf{b})|$ by induction on $n$. In the base case $n=2$, $\mathbf{b}$ is a single number $b_1$ and $\sort:[0,1]\to [0,1]$ is the identity map. We therefore have $d(\mathbf{b}) = 1 = |\sort^{-1}(\mathbf{b})|$ as claimed.\\

Now take $\mathbf{b} = (b_1,\dots,b_{n-1})\in B_n$ as above and assume that the claim holds for all smaller $n$. By induction hypothesis, there are exactly 
\[ d(b_1,\dots,b_{n-2}) = \prod_{j = 1}^{n-2} \#\{ i=0,\dots,j-1\,|\, b_j-b_i<1 \} \]
permutations of the coordinates $(b_0,\dots,b_{n-2})$ that lie in $A_{n-1}$. For any such permutation, let us insert $b_{n-1}$ into the permuted vector immediately following some coordinate $b_i$. The resulting vector lies in $A_n$ if and only if $b_{n-1}-b_i<1$. So for each permutation of $(b_1,\dots,b_{n-2})$ in $A_{n-1}$, there are exactly $\#\{i=0,\dots,n-2 \,|\, b_{n-1}-b_i<1\}$ places to insert $b_{n-1}$ that result in a vector lying in $A_n$. Every permutation of $\mathbf{b}$ that lies in $A_n$ can be obtained uniquely in this way, so by induction on $n$ we conclude that $|\sort^{-1}(\mathbf{b})| = d(\mathbf{b})$, completing the proof.
\end{proof}

\section{$m$-Dyck Path Combinatorics as $m\to\infty$}\label{sec:limits}

In this section we relate the combinatorics of continuous Dyck paths to $m$-Dyck paths. We will show that each of the continuous combinatorial objects defined in the previous section is a limit as $m\to\infty$ of its (normalized) $m$-Dyck path counterpart. This property motivated the definitions of the area, dinv, and bounce statistics on continous Dyck paths.\\

We say that a continuous Dyck path $D$ is $\frac1m$\textit{-integral} if each of the horizontal steps in $D$ is an integer multiple of $\frac1m$. Equivalently, $D\in\mathcal D^{(\cont)}_n$ is $\frac1m$-integral if its area vector lies in $A_n\cap \frac1m\Z^n$. The $\frac1m$-integral Dyck paths of height $n$ are in bijection with $m$-Dyck paths of height $n$, with the correspondence given by scaling horizontally by a factor of $m$ (see Figure \ref{fig:scaling} in the introduction).

\subsection{Area and Dinv as $m\to\infty$}

For a $\frac1m$-integral \ Dyck path $D$, we define $\overline{\area}_m(D)$ (resp. $\overline{\dinv}_m(D)$, $\overline{\bounce}_m(D)$) to be $\frac1m$ times the area (resp. dinv, bounce) of the corresponding $m$-Dyck path. We refer to these as the \textit{normalized $m$-area} (resp. \textit{dinv, bounce}) statistics of $D$.

\begin{example}\label{ex}
Let $D$ be the continuous Dyck path of height 3 with area vector $(0,1,1)$. The continuous statistics of $D$ are 
\[ \area(D)=2,\,  \dinv(D) = 1,\,  \bounce(D)=\frac12.\] 
This path $D$ is $\frac1m$-integral for any $m\geq 1$, corresponding to the $m$-Dyck path with area vector $(0,m,m)$. The normalized $m$-statistics of $D$ are
\[ \overline{\area}_m(D)=2,\hspace{2ex} \overline{\dinv}_m(D) = 1+\frac2m ,\hspace{2ex} \overline{\bounce}_m(D) =  
\begin{cases}
\frac 12 & m \text{ even},\\
\frac{m+1}{2m} & m \text{ odd}.
\end{cases} \]
\end{example}

The coincidence of the continuous $\area$ statistic and $\overline{\area}_m$ on $\frac1m$-integral Dyck paths holds in general.

\begin{lemma}\label{lemma:area}
For all $\frac1m$-integral Dyck paths $D$, we have $\area(D) = \overline{\area}_m(D).$
\end{lemma}

Example \ref{ex} shows that the situation for dinv and bounce is necessarily more complicated. In both of these cases, we will see that the normalized $m$-statistics converge uniformly over all $D$ to their continuous counterparts as $m\to\infty$.

\begin{lemma}\label{lemma:dinv}
\[ \lim_{m\to\infty} \max \left\{ \big|\dinv(D)-\overline{\dinv}_m(D)\big|\,\bigg|\, D\in \mathcal{D}^{(\cont)}_n\text{ $\frac1m$-integral} \right\} = 0 \]
\end{lemma}

\begin{proof}
Let $D$ be a $\frac1m$-integral Dyck path with area vector $(a_0(D),\dots,a_{n-1}(D))$, so that the corresponding $m$-Dyck path has area vector $(m a_0(D),\dots,m a_{n-1}(D)).$ We may compute $\overline{\dinv}_m(D)$ by the formula
\[\overline{\dinv}_m(D) =  \sum_{i<j}\frac{\sc_m(ma_i(D)-ma_j(D))}{m}, \]
while 
\[ \dinv(D) = \sum_{i<j} \sc(a_i(D)-a_j(D)). \]
Unwinding the definitions of $\sc_m$ and $\sc$, one checks that each pair of corresponding terms in the above sums differs by at most $\frac1m$. Therefore, 
\[ \big|\dinv(D)-\overline{\dinv}_m(D)\big| \leq \frac1m {n \choose 2}  \] 
for all $\frac1m$-integral $D\in \mathcal{D}^{(\cont)}_n$, which implies the claim.
\end{proof}

\subsection{$m$-Bounce Paths as $m\to\infty$}

Let us define a \textit{normalized $m$-bounce path} associated to any continuous Dyck path $D\in \mathcal{D}^{(m)}_n$. We modify the construction of the bounce parametrization in Section \ref{sec:contstats} to only allow north steps at times $t\in \frac1m\Z$. Let us describe the modified parametrization in more detail: \\

At time $t=0$, the normalized $m$-bounce path begins at the point $(0,0)$ and makes $v_0$ discrete (unit length) north steps so that it ends on a horizontal step of $D$. Then for all times $0<t<\frac1m$ the bounce path moves east continuously at speed $v_0$, covering a total horizontal distance of $\frac{v_0}{m}$. Inductively, assume that the parametrization of the bounce path has been defined for all $t<\frac im$ with north steps $v_0,\dots,v_{i-1}$ at the respective times $t=0,\dots,\frac{i-1}m$. At time $t=\frac im$ the bounce path makes $v_i$ (unit length) north steps so that it ends on a horizontal step of $D$. Then for all times $\frac im<t<\frac{i+1}m$ the bounce path moves east continuously at speed $v_i+\cdots+v_{i-m+1}$, which is the total length of the north steps traveled during the time interval $\frac{i}{m}-1< t \leq i$, covering a total horizontal distance of $\frac1m(v_i+\cdots+v_{i-m+1})$. The bounce path ends when it reaches the point $(n,n)$.\\

If the continuous Dyck path $D$ is $\frac1m$-integral, then the normalized $m$-bounce path of $D$ defined above coincides with the usual $m$-bounce path, after the usual horizontal rescaling by $m$. Indeed, the vertical steps $v_0,v_1,\dots$ in the parametrization above are the same as the vertical steps in the corresponding $m$-bounce path. Meanwhile the horizontal distance covered in the time interval $\frac im<t<\frac{i+1}m$, between the vertical steps $v_i$ and $v_{i+1}$, is exactly $\frac{h_i}m$. \\

Define the \textit{normalized $m$-bounce vector} of a continuous Dyck path $D$ to be \[ \overline{\mathbf{bounce}}_m(D) = (\underbrace{0,\dots,0}_{v_0},\underbrace{\frac 1m,\dots,\frac 1m}_{v_1},\dots,\underbrace{\frac im,\dots,\frac im}_{v_i},\dots)\in \Z^n. \]
In other words, the coordinates $\overline{\mathbf{bounce}}_m(D) = (\overline{b}_0(D),\dots,\overline{b}_{n-1}(D))$ are defined so that $\overline{b}_i(D)$ is the time at which the normalized $m$-bounce path takes its $i$th unit length vertical step.\\

One can show that as $m\to\infty$, these normalized $m$-bounce vectors converge uniformly to the continuous bounce vectors defined in Section \ref{sec:contstats}.

\begin{lemma}
\[ \lim_{m\to\infty} \max \left\{ \big|\mathbf{bounce}(D)-\overline{\mathbf{bounce}}_m(D)\big|\,\bigg|\, D\in \mathcal{D}^{(\cont)}_n \right\} = 0 \]
\end{lemma}

It is for this reason that we think of the bounce parametrization and bounce vector as the continuous analogue of the bounce vector for $m$-Dyck paths. When $D$ is $\frac1m$-integral, the sum of the coordinates of $\overline{\mathbf{bounce}}_m(D)$ coincides with the normalized bounce statistic $\overline{\bounce}_m(D)$. As a corollary, we obtain a similar result for the bounce statistics:
\begin{lemma}\label{lemma:bounce}
\[ \lim_{m\to\infty} \max \left\{ \big|\bounce(D)-\overline{\bounce}_m(D)\big|\,\bigg|\, D\in \mathcal{D}^{(\cont)}_n\text{ $\frac1m$-integral} \right\} = 0 \]
\end{lemma}

\subsection{The Bijections $\phi_m:\mathcal{D}^{(m)}_n\to\mathcal{D}^{(m)}_n$ as $m\to\infty$}

We briefly mention the limit interpretation of the map $T:\mathcal{D}^{(\cont)}_n\to\mathcal{D}^{(\cont)}_n$. In Section \ref{sec:DyckPaths} we reviewed the definition of the bijection $\phi_m:\mathcal{D}^{(m)}_n\to\mathcal{D}^{(m)}_n$, which is used to show the equivalence between the combinatorial definitions of the higher $q,t$-Catalan numbers. We claim that the map $T:\mathcal{D}^{(\cont)}_n\to\mathcal{D}^{(\cont)}_n$ defined in Section \ref{sec:T} is the uniform limit as $m\to\infty$ of the $\phi_m$'s.\\

Consider $T$ as a map $A_n\to A_n$ and $\phi_m$ as a bijection on the $\frac1m$-integral points of $A_n$. We claim that 
\[ \lim_{m\to\infty} \max \left\{ \big|T(\mathbf{a})-\phi_m(\mathbf{a})\big|\,\bigg|\, \mathbf{a}\in A_n\cap \frac1m\Z^n \right\} = 0. \]

The key observation is that for a $\frac1m$-integral path $D$, $\phi_m(D)$ is a path whose normalized $m$-bounce vector is equal to the sorted area vector of $D$. On the other hand, $T(D)$ is the path whose continuous bounce vector is equal to the sorted area vector of $D$. As $m\to\infty$, the normalized $m$-bounce paths of $\phi_m(D)$ converge to $T(D)$. Furthermore the maximum difference between a continuous Dyck paths and their normalized $m$-bounce path goes to zero as $m\to\infty$. These facts allow one to show the limit claim above.

\subsection{Higher $q,t$-Catalan Numbers as $m\to\infty$}\label{sec:proof}

The following theorem is the main result of this paper, realizing the $q,t$-Catalan measure $\mu_n$ as a limit of higher $q,t$-Catalan numbers. We use the dinv/area definitions of the higher $q,t$-Catalan numbers and $q,t$-Catalan measures. 

\begin{theorem}\label{thm:main}
For all $n\geq 1,$ the $q,t$-Catalan measure $\mu_n$ is equal to the weak limit
\begin{equation} \lim_{m\to\infty}\left(\frac1{m^{n-1}}\sum_{D\in\mathcal D^{(m)}_n}\delta_{\left(\frac{\dinv_m(D)}{m},\frac{\area_m(D)}{m}\right)}\right). \end{equation}
\end{theorem}

\begin{proof}
Fix $n\geq 1$, and let $h:\R^2\to\R$ be a bounded continuous function. We will show that the integrals of $h$ against the sequence of discrete measures in the theorem statement converge to $\int h\, \mathrm{d}\mu_n$. All of the relevant measures are supported in the convex set $[0,{n-1\choose 2}]^2\subseteq \R^2$, so we may assume without loss of generality that $h$ is compactly supported. In particular, we take $h$ to be uniformly continuous. \\

We identify $\mathcal{D}_n^{(\cont)}\leftrightarrow A_n$, and consider $A_n\subseteq \{0\}\times\R^{n-1}$ as a full-dimensional polytope in $\R^{n-1}$. Under this identification, we define $\pi$ to be the map $(\dinv,\area):A_n\to \R^2$. Similarly, let $\pi_m$ be the map $(\overline{\dinv}_m,\overline{\area}_m):A_n\cap \frac1m\Z^{n-1}\to \R^2$. By the correspondence $\mathcal{D}_n^{(m)}\leftrightarrow A_n\cap \frac1m \Z^{n-1}$ and the definition of the normalized statistics $\overline{\dinv}_m,\overline{\area}_m$, the discrete measures in the theorem statement can be rewritten as
\[ \frac1{m^{n-1}}\sum_{D\in\mathcal D^{(m)}_n}\delta_{\left(\frac{\dinv_m(D)}{m},\frac{\area_m(D)}{m}\right)} = \frac{1}{m^{n-1}}\sum_{\mathbf{a}\in A_n\cap \frac1m\Z^{n-1}} \delta_{\pi_m(\mathbf{a})}. \]
Integrating $h$ against these measures, we obtain
\[ \int h \, \mathrm{d}\left( \frac{1}{m^{n-1}}\sum_{\mathbf{a}\in A_n\cap \frac1m\Z^{n-1}} \delta_{\pi_m(\mathbf{a})} \right) = \sum_{\mathbf{a}\in A_n\cap \frac1m\Z^{n-1}} \frac{h(\pi_m(\mathbf{a}))}{m^{n-1}}. \]

By Lemmas \ref{lemma:area} and \ref{lemma:dinv}, the functions $\pi_m$ converge uniformly to $\pi$ as $m\to \infty$, in the sense that 
\[ \lim_{m\to\infty} \max \left\{ \big| \pi_m(\mathbf{a})-\pi(\mathbf{a})\big| \,\bigg|\, \mathbf{a}\in A_n\cap \frac1m\Z^{n-1} \right\} = 0. \]
Our assumption of uniform continuity of $h$ allows us to conclude that
\[ \lim_{m\to\infty} \max \left\{ \big| h(\pi_m(\mathbf{a}))-h(\pi(\mathbf{a}))\big| \,\bigg|\, \mathbf{a}\in A_n\cap \frac1m\Z^{n-1} \right\} = 0. \]
Since the number of points $\mathbf{a}\in A_n\cap \frac1m\Z^{n-1}$ is a polynomial of degree $n-1$ in $m$, the sums
\[  \sum_{\mathbf{a}\in A_n\cap \frac1m\Z^{n-1}} \frac{h(\pi_m(\mathbf{a}))-h(\pi(\mathbf{a}))}{m^{n-1}} = \sum_{\mathbf{a}\in A_n\cap \frac1m\Z^{n-1}} \frac{h(\pi_m(\mathbf{a}))}{m^{n-1}} \,- \sum_{\mathbf{a}\in A_n\cap \frac1m\Z^{n-1}} \frac{h(\pi(\mathbf{a}))}{m^{n-1}} \]
also converge to zero as $m\to\infty$. The final sum above is a Riemann sum approximation of the integral $\int (h\circ\pi)\,\mathrm{d}\lambda_n$, where $\lambda_n$ is the Lebesgue measure restricted to the full-dimensional polytope $A_n\subseteq \R^{n-1}$. Since $h\circ\pi$ is continuous, these Riemann sums converge to the integral $\int(h\circ \pi)\,\mathrm{d}\lambda_n = \int h \,\mathrm{d} (\pi_*\lambda_n) = \int h\, \mathrm{d} \mu_n$ as $m\to\infty$. The integrals of $h$ against the discrete measures in the statement therefore converge to this integral as well, which completes the proof.
\end{proof}

Since the higher $q,t$-Catalan numbers are known to be symmetric, the discrete measures in the limit appearing in Theorem \ref{thm:main} are symmetric. This implies the following symmetry property for $\mu_n$.

\begin{corollary}\label{cor:symmetry}
For all $n\geq 1$ the $q,t$-Catalan measure $\mu_n$ is symmetric about the line $y=x$.
\end{corollary}

\section{Connection to Hilbert Schemes}\label{sec:geometric}

\subsection{Hilbert Schemes and Higher $q,t$-Catalan Numbers}

We first reproduce the geometric setup for the $q,t$-Catalan numbers described by Haiman \cite{Hai}.\\

Let $\H^n$ denote the Hilbert scheme of $n$ points in $\C^2$. This Hilbert scheme is a smooth irreducible variety of dimension $2n$ that parametrizes length $n$ subschemes of $\C^2$. Let $\H^n_0\subseteq \H^n$ denote the (reduced) subscheme parametrizing subschemes supported at the origin in $\C^2$. The punctual Hilbert scheme, $\H^n_0$, is an irreducible and reduced variety of dimension $n-1$. Equivalently, $\H^n$ parametrizes ideals $I\subseteq \C[x,y]$ such that $\C[x,y]/I$ is an $n$-dimensional complex vector space, and $\H^n_0$ parametrizes such ideals $I$ satisfying the additional constraint that $\sqrt{I}=(x,y)\subseteq \C[x,y]$.\\

There is a two-dimensional torus, $T=(\C^*)^2$, that acts naturally on $\H^n$. The action moves the support of the subschemes around by scaling coordinates by the corresponding coordinates of $(\C^*)^2$, but some care must be taken when the subschemes are nonreduced. The punctual Hilbert scheme is set theoretically fixed by this torus action, so the $T$ action can be restricted to $\H^n_0$. \\

Haiman showed that $\H^n$ admits the explicit description $\H^n\simeq \Proj \bigoplus_{m\geq 0}A^m$, where $A\subseteq \C[x_1,y_1,\dots,x_n,y_n]$ is the set of alternating polynomials under the symmetric group action permuting the variables in blocks $(x_i,y_i)\leftrightarrow (x_j,y_j)$. This identification equips $\H^n$ with an ample line bundle. We denote the restriction of this line bundle to $\H^n_0$ by $\Oo(1)$, and its tensor powers by $\Oo(m)$ for all $m\geq1$ .\\

The torus action on $\H^n_0$ extends to a compatible action on the line bundles $\Oo(m)$. This induces a linear action of $T$ on the spaces of global sections $H^0(\H^n_0,\Oo(m))$. Linear torus actions split into direct sums of characters, so we may consider $H^0(\H^n_0,\Oo(m))$ as a bigraded vector space. Let $H^0(\H^n_0,\Oo(m))_{u,v}$ denote the component corresponding to the character $(u,v)\in \Z^2$. The \textit{geometric higher $q,t$-Catalan number} $GC^{(m)}_n(q,t)$ encodes the dimensions of the bigraded pieces of these vector spaces by the formula
\[ GC^{(m)}_n(q,t) = \sum_{u,v}q^u t^v \dim H^0(\H^n_0,\Oo(m))_{u,v} \]

\begin{theorem}[\cite{GHai,Hai,Hai2,M}]\label{thm:geom}
For all $n,m\geq 1$, we have
\[C^{(m)}_n(q,t) = GC^{(m)}_n(q,t)\]
\end{theorem}

The proof of this result passes through several other definitions of the higher $q,t$-Catalan numbers, which we briefly explain. We borrow notation and terminology from \cite{LLL}, which contains a more detailed summary of the known equivalences (at that time) between various definitions of higher $q,t$-Catalan numbers. \\

Haiman \cite{Hai,Hai2} showed that $GC^{(m)}_n(q,t)$ agrees with the original definition of the higher $q,t$-Catalan numbers \cite{GHai} as a sum of rational functions indexed by partitions. The original definition in \cite{GHai} was also known to agree with the polynomial $\langle \nabla^m e_n,e_n\rangle$, where $\nabla$ is a certain operator related to Macdonald polynomials. \\

In the case $m=1$, Garsia and Haglund \cite{GHag} showed that $\langle \nabla e_n,e_n\rangle$ agrees with the combinatorially defined polynomials $C_n(q,t)$. The result for $m>1$ is a special case of a much more general result proved by Mellit \cite{M}, known as the compositional $(km,kn)$-shuffle conjecture \cite{BGLX}. This result is a generalization of the (higher) shuffle conjecture formulated in \cite{HHLRU} which amounts to a combinatorial formula for the entire polynomial $\nabla^m e_n$. A consequence is that the combinatorially defined higher $q,t$-Catalan numbers, $C^{(m)}_n(q,t)$, agree with $\langle \nabla^m e_n,e_n\rangle$.

\subsection{Duistermaat-Heckman Measure}

Let $X$ be a $d$-dimensional variety with a line bundle $\mathcal L$. Suppose that $X$ is equipped with a $k$-dimensional torus action, and that the action extends compatibly to $\mathcal L$. Then there is a linear action of $T$ on the spaces of global sections of $\mathcal L^{\otimes m}$, which is equivalent to a $\Z^k$-grading on these spaces. For $\mathbf{v}\in\Z^k$ we denote the degree $\mathbf{v}$ component by $H^0(X,\mathcal{L}^{\otimes m})_{\mathbf{v}}$. The \textit{Duistermaat-Heckman measure} of the triple $(X,\mathcal L,T)$ is defined to be the weak limit of measures on $\R^k$,
\[ \mathrm{DH}(X,\mathcal L,T) = \lim_{m\to\infty} \left( \frac{1}{m^d} \sum_{\mathbf{v}\in \Z^k} \delta_{\frac{\mathbf{v}}{m}}\dim H^0(X,\mathcal L^{\otimes m})_{\mathbf{v}} \right). \]

Brion and Procesi \cite{BP} prove that this limit exists for all such triples $(X,\mathcal L,T)$. They also show that the resulting measure is piece-wise polynomial on $\R^k$.\\

In the present case, the relevant triple consists of the punctual Hilbert scheme with line bundle $\Oo(1)$ and two-dimensional torus $T$ acting as described in the previous section. With this setup, the Duistermaat-Heckman measure of the punctual Hilbert scheme is defined as the weak limit
\begin{equation}\label{def:DH} \mathrm{DH}(\H^n_0,\Oo(1),T) = \lim_{m\to\infty} \left( \frac1{m^{n-1}}\sum_{u,v}\delta_{\left(\frac{u}m,\frac{v}m\right)}\dim H^0(\H^n_0,\Oo(m))_{u,v}\right). \end{equation}

The equality $C^{(m)}_n(q,t)=GC^{(m)}_n(q,t)$ asserted in Theorem \ref{thm:geom} implies that the limit above coincides with the limit appearing in Theorem \ref{thm:main}. This allows for the following geometric reinterpretation of Theorem \ref{thm:main}.

\begin{theorem}\label{thm:DH}
For all $n\geq 1$, the $q,t$-Catalan measure $\mu_n$ is equal to the Duistermaat-Heckman measure $\mathrm{DH}(\H^n_0,\Oo(1),T)$.
\end{theorem}

There are geometric formulas for computing Duistermaat-Heckman measures. Example 1.2.6 in \cite{K} describes the geometric setup for another approach to these measures for $\H^n_0$ in the case $n=4$, as well as an alternate interpretation of the figure in Example \ref{ex:4measure}. To our knowledge though, Theorem \ref{thm:DH} is the first combinatorial formula for the Duistermaat-Heckman measure of $\H^n_0$.

\pagebreak

\end{document}